\documentclass[12pt,reqno,sort]{elsarticle}

\usepackage[colorlinks, citecolor=blue,linkcolor=blue, anchorcolor=blue, pagebackref]{hyperref}
\usepackage{caption,amsmath,amssymb,amsfonts,amsthm,latexsym,graphicx,multirow,color,enumerate}
\usepackage{subfigure}
\usepackage{graphicx,graphics,epstopdf}
\graphicspath{{figures/}}
\usepackage{natbib}
\usepackage[colorlinks, citecolor=blue,linkcolor=blue, anchorcolor=blue, pagebackref]{hyperref}
\usepackage[all]{xy}
\usepackage{mathrsfs}
\usepackage{longtable}
\usepackage{colortbl}
\usepackage{graphics}
\usepackage{ulem}
\usepackage{booktabs}

\def\ov{\overline} 
\def\l{\langle} \def\r{\rangle}

\def\ZZ{{\sf Z}}

\def\mod{{\sf mod~}}

\def\Aut{{\sf Aut}} 
\def\Out{{\sf Out}}
\def\Cos{{\sf Cos}}

 \def\soc{{\sf soc}}

\def\D{{\rm D}} 
\def\S{{\rm S}} 
 \def\M{{\rm M}}

\def\C{{\bf C}}\def\N{{\bf N}}

\def\calO{{\mathcal O}}

\def\Ga{{\it \Gamma}} 
\def\Sig{{\it \Sigma}}

\def\a{\alpha} \def\b{\beta}

\def\GF{{\rm GF}} \def\GL{{\rm GL}}

\def\A{{\rm A}}

\def\PSL{{\rm PSL}}  \def\PGL{{\rm PGL}}
\def\GL{{\rm GL}} \def\SL{{\rm SL}}
\def\AGL{{\rm AGL}}

  \def\D{{\rm D}}

\def\char{{\sf \ char\ }}

\def\calM{{\mathcal M}}

\def\M{{\sf M}}

\def\soc{{\rm soc}}

\def\le{\leqslant}

\def\RotaMap{{\sf RotaMap}}
\def\BiRoMap{{\sf BiRoMap}}

\def\core{{\sf core}}

\def\calO{{\mathcal O}}
\def\a{\alpha}
\def\ov{\overline}
\def\GF{{\rm GF}}

\newtheorem{theorem}{Theorem}[section]%
\newtheorem{lemma}[theorem]{Lemma}%
\newtheorem{corollary}[theorem]{Corollary}%
\newtheorem{proposition}[theorem]{Proposition}%
\newtheorem{definition}[theorem]{Definition}%
\newtheorem{problem}[theorem]{Problem}%
\newtheorem{example}[theorem]{Example}%
\newtheorem{hypothesis}[theorem]{Hypothesis}%

\def\qed{{\hfill$\Box$\smallskip}
	\medbreak}

\begin{document}
	
\begin{frontmatter}
	\title{Hall Skew-morphisms and Hall Cayley maps of finite groups \tnoteref{funding}} 
	\tnotetext[funding]{This work was supported by NSFC grants 12461061, and 11931005.}
	\author[SUSTech]{Wendi Di}
	\ead{wendyjj@zju.edu.cn}
	\author[ZZu]{Zheng Guo\corref{cor1}}
	\ead{12131227@mail.sustech.edu.cn}
	\cortext[cor1]{Corresponding author}
	\author[SUSTech]{Cai Heng Li}
	\ead{lich@sustech.edu.cn}
	
	\affiliation[ZZu]{organization={Zhengzhou University},
		addressline={No. 100, Kexue Avenue}, 
		city={Zhengzhou},
		postcode={450001}, 
		state={Henan},
		country={China}}
	\affiliation[SUSTech]{organization={Southern University of Science and Technology},
		addressline={1088 Xueyuan Avenue}, 
		city={Shenzhen},
		postcode={518055}, 
		state={Guangdong},
		country={China}}
	\begin{abstract}
		A characterization is given of finite groups $H$ that have skew-morphisms of order coprime to the order $|H|$, and their skew-morphisms. A complete classification is then given of the automorphism groups and the underlying graphs of vertex-rotary core-free Hall Cayley maps.
	\end{abstract}
	
	
\end{frontmatter}
	\section{Introduction}
	
	For a group $H$, a {\it skew-morphism} of $H$ is a permutation $\rho$ on $H$ such that
	$$\rho(1)=1\text{ and }\rho(gh)=\rho(g)\rho^{\pi(g)}(h),$$
	where $g,h\in H$, and $\pi$ is an integer function on $H$.
	In particular, when $\pi(g)=1$ for each $g\in H$, the skew-morphism $\rho$ is actually an automorphism of $H$, called a {\it trivial skew-morphism}.
	The concept of skew-morphism was introduced by Jajcay and \v{S}ir\'{a}\v{n} in \cite{Jajcay2003}, in order to investigate regular Cayley maps.
	There is an equivalent definition of skew-morphism in the version of group theory, refer to \cite[Theorem 1]{Jajcay2003}.
	
	\begin{definition}\label{def:skew-morphism}
		{\rm
			For a group $H$, if there exists a group $G$ such that
			\[G=HK,\ \mbox{where $H\cap K=1$ and $K$ is cyclic and core-free in $G$,}\]
			then each generator of $K$ is called a {\it skew-morphism} of $H$.
			In this case, $G=HK$ is called a {\it skew product} of $H$ and $K$.
		}
	\end{definition}
	
Here we have some obvious examples for non-trivial skew-morphisms:
a symmetric group $\S_n$ has a skew-morphism of order $n+1$ since $\S_{n+1}=\S_n\ZZ_{n+1}$;
a dihedral group $\D_8$ has a non-trivial skew-morphism of order 3 as $\S_4=\D_8\ZZ_3$;
for an odd prime $p$, a dihedral group $\D_{2p}$ has a non-trivial skew-morphism of order $p$ since $\ZZ_p\wr\S_2=\D_{2p}\ZZ_p$.

A central problem on skew-morphisms is the determination of skew-morphisms for given families of finite groups.
The problem remains challenging, and is unsettled even for some very special families of groups although a lot of efforts have been made, refer to \cite{BachratyJajcay2017,ConderJajcayTucker2016,ConderTucker2014,DuHu2019,KN2011,KovacsNedela2017} for partial results on skew-morphisms of cyclic groups;
see \cite{KovacsMarusicMuzychuk2013,Zhang2015a,Zhang2015b,KovacsKwon2016,MR4579715} for partial results on the skew-morphisms of dihedral groups;
see \cite{DuLuoYuZhang24,DuYuLuo23} for the skew-morphisms of elementary abelian $p$-groups $\ZZ_p^n$.
Recently, the skew-morphisms of finite monolithic groups are characterized in \cite{BCMG2022}, and the skew-morphisms of finite nonabelian characteristically simple groups are characterized in \cite{ChenDuLi22}.
	
In this paper, we characterize finite groups $H$ that have skew-morphisms of the order coprime to the order $|H|$ and their skew-morphisms.
The examples come mainly from linear groups $T$ of prime dimension acting on 1-subspaces, which provides a factorization $T=HK$, where $T=\PSL(d,q)$ with $\gcd(d,q-1)=1$, and
	\[\begin{array}{l}
		H=\AGL(d-1,q)=q^{d-1}{:}\GL(d-1,q),\ \ \mbox{the stabilizer of a 1-subspace,}\\
		K=\ZZ_{\frac{q^d-1}{ q-1}}, \ \mbox{a Singer cycle}.\\
	\end{array}\]

To state our main results, we make the following hypothesis.
	
\begin{hypothesis}\label{hypo-1}
{\rm
Let $T$ be an almost simple group, associated with a parameter $e(T)$ and a factorization $T=HK$, as in the following table:

\begin{table}[h!]
\centering
\begin{tabular}{|ccccc|}
        \hline
        $T$ & $e(T)$ & $H$ & $K$ & Remark \\
        \hline
        $\A_p$, $\S_p$ & $p$ & $\A_{p-1}$,\ $\S_{p-1}$ & $\ZZ_p$ & \text{$p$ prime}\\ \hline
       \multirow{2}{*}{ \text{$\PSL(d,q){:}\l\phi\r$}} & \multirow{2}{*}{$\dfrac{q^d-1}{q-1}$ }& \multirow{2}{*}{$\AGL(d-1,q){:}\l\phi\r$ }& \multirow{2}{*}{$\ZZ_{\frac{q^d-1}{q-1}}$} & \multicolumn{1}{c|}{$d$ prime, $\gcd(d,q-1)=1$,} \\
        & & & & \multicolumn{1}{c|}{$\phi$  a field automorphism}\\ \hline
        $\PSL(2,11)$ & $11$ & $\A_5$ & $\ZZ_{11}$ & \\
        $\M_{11}$ & $11$ & $\M_{10}$ & $\ZZ_{11}$ & \\
        $\M_{23}$ & $23$ & $\M_{22}$ & $\ZZ_{23}$ & \\
\hline
\end{tabular}
\caption{}
\label{tab:1}
\end{table}

}
\end{hypothesis}
		
The first main result of this paper is stated in the following theorem.
		
\begin{theorem}\label{skewG}
Let $G=HK$ be a group factorization such that  $H$ is a Hall subgroup  and $K$ is cyclic, and  let $N$ be the core of $H$ in $G$.
Then either
\begin{itemize}
\item[\rm(1)] $G=N.(K{:}\calO)$, where $H=N.\calO$ and $\calO\leqslant\Aut(K)$, or
				
\item[\rm(2)] $G=N.(T_1\times \dots\times T_r\times K_0).\calO$,
where $\gcd(|T_i|,e(T_j))=1$ for any $i\not=j$,  $\calO\leqslant\Out(T_1)\times\dots\times\Out(T_r)\times \Aut(K_0)$,  and $T_i=H_iK_i$ is a simple group satisfying Hypothesis~$\ref{hypo-1}$ such that
\[\mbox{$H=N.(H_1\times\dots\times H_r).\calO$, and $K=K_0\times K_1\times\dots\times K_r$.}\]
\end{itemize}
\end{theorem}

We remark that the numerical condition appeared in Theorem~\ref{skewG}\,(2):
\[\mbox{$\gcd(|T_i|,e(T_j))=1$ for any distinct values $i,j\in\{1,2,\dots,r\}$}\]
is very restricted.
For instance, $\{T_1,\dots,T_r\}$ contains at most one alternating group or symmetric group.
However, it is shown that there is no upper bound for the number $r$ of the direct factors $T_i$.

\begin{corollary}\label{cor:r-nobound}
For any positive integer $r$, there exist $r$ linear groups $T_i=\PSL(d_i,q_i)$ with $1\leqslant i\leqslant r$ such that
$G=T_1\times\dots\times T_r$ is a skew-product $G=H\l\rho\r$ with $\gcd(|H|,|\rho|)=1$.
\end{corollary}

In the proof of Corollary~\ref{cor:r-nobound}, examples for $G=T_1\times\dots\times T_r$ with $|T_1|<\dots<|T_r|$ are constructed for arbitrarily large $r$.
However, the known examples are such that
\[\mbox{$|T_1|\to\infty$ when $r\to\infty$.}\]
This leads to a natural problem.

\begin{problem}\label{q:r-bound}
{\rm
Characterize linear groups $T_i=\PSL(d_i,q_i)$ with $1\leqslant  i\leqslant r$ with $|T_1|<\dots<|T_r|$ and $|T_1|$ 
upper-bounded such that $G=T_1\times\dots\times T_r=H\l\rho\r$ with $\gcd(|H|,|\rho|)=1$.
}
\end{problem}

A skew-morphism $\rho$ of a group $H$ is called a {\it Hall skew-morphism} if the order $|H|$ is coprime to the order $|\rho|$.
Then Theorem~\ref{skewG} has the following consequnce.

\begin{theorem}\label{thm:skew-G}
A finite group $H$ has a Hall skew-morphism $\rho$ if and only if
\[H=N.(H_0\times H_1\times\dots\times H_r).\calO,\]
where $\calO$ is as in Theorem~\ref{skewG}, $\gcd\left(|N||\calO|,|\rho|\right)=1$, and either $H_i=1$ or
\begin{enumerate}
\item[\rm(i)] $(H_0,\ell_0)=(\A_{p-1},p)$, $(\A_5,11)$, $(\M_{10},11)$, $(\M_{22},23)$, $(\A_5\times \A_6,11\times 7)$, $(\M_{10}\times \A_6,11\times 7)$, $(\M_{22}\times \A_{12},23\times13)$, $(\M_{22}\times \A_{16},23\times17)$, or $(\M_{22}\times \A_{18},23\times19)$;

\item[\rm(ii)] $(H_i,\ell_i)=(\AGL(d_i,q_i),{q_i^{d_i+1}-1\over q_i-1})$ with $1\leqslant i\leqslant r$.
\end{enumerate}
Further, $\gcd(\ell_i|H_i|,\ell_j)=1$  for any distinct $i,j\in\{1,2,\dots,r\}$, and $|\rho|=\ell_0\ell_1\dots\ell_r$. \end{theorem}

We observe that the triples $(T,H,K)$ listed in Hypothesis~\ref{hypo-1} with $H$ solvable are as follows:
		\[(\SL(3,2),\S_4,7), (\PSL(3,3),\AGL(2,3),13), (\SL(2,2^f),\AGL(1,2^f),2^f+1).\]
This leads to the following consequence of Theorem~\ref{skewG}, which determines Hall skew-morphisms of finite solvable groups.

\begin{corollary}\label{cor:solvable}
Let $G=HK$ be a factorization such that $H$ is a solvable Hall subgroup of $G$ and $K$ is cyclic.
Let $N$ be the core of $H$ in $G$.
Then either
\begin{itemize}
\item[\rm(1)] $G=N.(K{:}\calO)$, and $H=N.\calO$ with $\calO\leq \Aut(K)$ abelian, or
				
\item[\rm(2)] $G=N.(E\times K_0).\calO$, where $K_0<K$ and $\calO\leqslant \Out(E)\times\Aut(K_0)$, and either
\begin{itemize}
\item [\rm(i)] $E=\SL(3,2)$, $\PSL(3,3)$, or $\SL(2,2^f)$, or

\item [\rm(ii)] $E\lhd \SL(3,2)\times \PSL(3,3)\times \SL(2,2^f)$, where $f\equiv 2,4\pmod 6$.
\end{itemize}
\end{itemize}
\end{corollary}
		
Next, we apply Theorem~\ref{skewG} to study a class of highly symmetric maps.

Let $\calM=(V,E,F)$ be a map, with vertex set $V$, edge set $E$ and face set $F$.
A {\it flag} $(\a,e,f)$ of a map is an incident triple of vertex $\a$, edge $e$ and face $f$.
A map $\calM$ is called {\it regular} if the automorphism group $\Aut\calM$ is regular on the flag set of $\calM$.
Regular maps have the highest symmetry degree, and slightly lower symmetrical maps include arc-transitive maps and edge transitive maps, which have received considerable attention in the literature, see \cite{Reg-maps,Edge-t-maps,Li-2008} and references therein.
In this paper, we study two classes of arc-transitive maps, defined below.

For an edge $e = [\a, e, \a']$, the two faces of $\calM$ incident with $e$ is denoted by $f$ and $f'$.
For a subgroup $G\leqslant\Aut\calM$, the map $\calM$ is called {\it $G$-vertex-rotary} if $G$ is arc-regular on $\calM$ and the vertex stabilizer $G_\a=\l \rho\r$ is cyclic.
In this case, $G$ contains an involution $z$ such that $G=\l \rho,z\r$.
We call the pair $(\rho,z)$ a {\it rotary pair} of $G$.
With such a rotary pair $(\rho,z)$, we have a {\it coset graph}
		\[\Ga=\Cos(G,\l \rho\r,\l \rho \r z\l \rho\r),\]
which has vertex set $V=[G:\l \rho\r]$ such that
\[\mbox{$\l \rho\r x$ and $\l \rho\r y$ are adjacent if and only if $yx^{-1}\in \l \rho\r z\l \rho\r$.}\]
The vertex stabilizer $G_\a=\l\rho\r$ acts regularly on $E(\a)$, the edge set incident with $\a$.
The graph $\Cos(G,\l\rho\r,\l \rho\r z\l \rho\r)$ has {\it vertex-rotary} embeddings, which are divided into
two different types according to the action of $z$ on the two faces $f,f'$ which are incident with the edge $e$, see \cite{LPS}.
That is to say, either
\begin{itemize}
\item $z$ interchanges $f$ and $f'$, and $\calM$ is {\it $G$-rotary} (also called {\it orientably regular}),
denoted by $\RotaMap(G,\rho,z)$, or

\item $z$ fixes both $f$ and $f'$, and $\calM$ is {\it $G$-bi-rotary }, denoted by $\BiRoMap(G,\rho,z)$.
\end{itemize}

A map $\calM$ is called a {\it Cayley map} of a group $H$ if $\Aut\calM$ contains a subgroup which is isomorphic to $H$ and regular on the vertex set $V$.
The study of Cayley maps has been an active research topic in algebraic and topological graph theory for a long time, refer to \cite{Li-2006,Cay-maps,Balanced-Cay,Balanced-Cay-2} and reference therein.
As an application of Theorem~\ref{skewG}, we focus us on a special class of Cayley maps.
A Cayley map $\calM$ of $H$ is called a {\it Hall Cayley map} of $H$ if $H$ is isomorphic to a Hall subgroup of $\Aut\calM$, and called a {\it core-free Cayley map} if $H$ is core-free in $\Aut\calM$.

The following theorem presents a classification for the automorphism groups and underlying graphs of vertex-rotary maps which are core-free Hall Cayley maps.
We first determine almost simple groups which are vertex-rotary on a Hall Cayley map, and then decompose the general case into the almost simple ones by `direct product' and `bi-direct product', defined before Lemma~\ref{lem:productmap}.
The classification is stated in the following theorem.

\begin{theorem}\label{thm:maps}
Let $\calM$ be a $G$-vertex-rotary map.
Then $\calM$ is a core-free Hall Cayley map if and only if
\[G=\left((T_1\times\dots\times T_s){:}\l z_1\dots z_s\r\right)\times T_{s+1}\times\dots\times T_r,\ \mbox{for some $0\le s\le r$},\]
where $T_i$ is a simple group  in Hypothesis~$\ref{hypo-1}$ with $\gcd(|T_i|,e(T_j))=1$ for any $i\not=j$, and $z_i\le \Out(T_i)$ is  of order $2$.
Moreover, $\calM$ has underlying  graph $$\Ga=\left(\Ga_1\times_{\rm{bi}}\Ga_2\times_{\rm{bi}}\cdots\times_{\rm{bi}}\Ga_s\right)\times \left(\Ga_{s+1}\times\cdots\times\Ga_r\right),$$
where $\Ga_i=\Cos(T_i{:}\l z_i\r,\l\rho_i\r,\l\rho_i\r z_i\l\rho_i\r)$ for $i\leqslant s$, or  $\Cos(T_j,\l\rho_j\r,\l\rho_j\r z_j\l\rho_j\r)$ for $j>s$.
\end{theorem}

In the subsequent article \cite{Hall-maps}, a characterization and enumeration will be given for vertex-rotary core-free Hall Cayley maps.


\section{Hall factorizations and skew-morphisms}
		
In this section, we prove Theorems~\ref{skewG} and \ref{thm:skew-G} and their corollaries.

We first establish some useful lemmas.
A group factorization $G=HK$ is called a {\it Hall factorization} if $H$ or $K$ is a Hall subgroup of $G$.
The following lemma states that a Hall factorization can be inherited by its subnormal subgroups.
(Recall that a subgroup $M<G$ is a {\it subnormal subgroup} of $G$ if there exist subgroup sequence $M=M_n\lhd M_{n-1}\lhd\dots\lhd M_1\lhd G$.)
		
		\begin{lemma}\label{subnormal}
			Let $G=HK$ be a Hall factorization and $M$ a subnormal subgroup of $G$.
			Then $M=(M\cap H)(M\cap K)$ is a Hall factorization.
		\end{lemma}
		
\begin{proof}
Since $M$ is a subnormal subgroup of $G$, we can assume that $M=M_n\lhd M_{n-1}\lhd\dots\lhd M_1\lhd G$ for some positive integer $n$.
(The proof will be proceeded by induction on $n$.)
			
			For $n=1$, we have $M\lhd G$.
			Since $G=HK$ is a Hall factorization, we have $H\cap K=1$.
			Noting that $(M\cap H)(M\cap K)\le M$ and
			\[
			(M\cap H)\cap (M\cap K)=M\cap H\cap K=1,
			\]
			we only need to show
			\[
			|M|=|M\cap G|\cdot |M\cap K|.
			\]
Set $\ov{G}=G/M$, $\ov{H}=HM/M$, and $\ov{K}=KM/M$. Then $\ov{G} =\ov{H}\ov{K}$ is also a Hall factorization.
			From $\ov{H}\cong H/(M \cap H)$ and $\ov{K} \cong K/(M\cap K)$, we obtain
			\[
			|\ov{G}|=|\ov{G}|\cdot|\ov{K}|=\frac{|G|}{|M\cap H|}\cdot\frac{|K|}{|M\cap K|}.
			\]
			On the other hand, since $|G|=|H|\cdot |K|$, we have $|\ov{G}|=\frac{|H|\cdot|K|}{|M|}$.
			It follows that
			\[
			|M|=|M\cap H|\cdot |M\cap K|,
			\]
			and thus $M=(M\cap H)(M\cap K)$.
			
Now suppose $n>1$ and that $M_{n-1}=(M_{n-1}\cap H)(M_{n-1}\cap K)$, which is a Hall factorization by induction assumption.
			Since $M_n\lhd M_{n-1}$, by the argument above we have
			\[
			M_{n}=(M_{n}\cap (M_{n-1}\cap H))(M_{n}\cap (M_{n-1}\cap K))
			=(M_{n}\cap H)(M_{n}\cap K).
			\]
			Therefore, $M=(M\cap H)(M\cap K)$, and the proof is completed.
		\end{proof}

		\begin{lemma}\label{lem:quotient}
			Let $G$ be a finite group with a Hall factorization $G=HK$ and $N\lhd G$.
			Set $\ov G=G/N$ such that neither $H$ nor $K$ is contained in $N$, $\ov H=HN/N$ and $\ov K=KN/N$.
			Then $\ov G$ has a Hall factorization $\ov{G}=\ov{H}\,\ov{K}$.
		\end{lemma}
		\begin{proof}
			Since $N\lhd G$, we have $N\cap H\lhd H$ and $N\cap K\lhd K$.
			Thus $\ov{H}\cong H/(N \cap H)$ and $\ov{K} \cong K/(N\cap K)$.
			Since $G=HK$, we have $\ov{G} = G/N =\ov{H}\,\ov{K}$, and $\gcd(|\ov H|,|\ov K|)=1$ as
			$\gcd(|H|,|K|)=1$.
			So $\ov{G}=\ov{H}\,\ov{K}$ is a Hall factorization.
		\end{proof}
		Next, we consider solvable \text{groups $G$}.
		
\begin{lemma}\label{lem:solvable}
Let $G=HK$ be a solvable group, where $H$ is a Hall subgroup of $G$, and $K$ is cyclic.
Then $G=N.(K{:}\calO)$, and $H=N.\calO$, where $N$ is the core of $H$ in $G$, and $\calO\leqslant\Aut(K)$ is abelian.
\end{lemma}
		
		\begin{proof}
			In order to prove the lemma, we may assume that $H$ is not normal in $G$.
			Let $N$ be the core of $H$ in $G$, and let $\ov G=G/N$.
			Then $\ov G=\ov H\,\ov K$ is a Hall factorization by Lemma \ref{lem:quotient}, and $\ov H$ is core-free in $\ov G$.
			Thus, to complete the proof, we may assume that $N=1$.
			
			Let $F$ be the Fitting subgroup of $G$, and let $\calO=G/F$.
			Since $H$ is core free, we have $F\cap H=1$.  
If not, there is some prime $p\mid |H|$ such that $O_p(G)\le F\cap H$, a contradiction.
			Let $\pi=\pi(H)$, the set of prime divisors of the order $|H|$.
			Then $F$ is a $\pi'$-subgroup of $G$, and thus $F\leqslant K$. 	
It follows that $F=K$ and $\calO\leqslant\Aut(K)$ is abelian. 
Since $K$ is a Hall normal subgroup of $G$, we have $G=K{:}O$.
		\end{proof}

We recall that a permutation group is called a {\it c-group} if it has a regular cyclic subgroup. Almost simple c-groups are determined in \cite{li2012cyclic}.
		
\begin{lemma}\label{2.6}
Let $T$ be a nonabelian almost simple group which has a non-trivial Hall factorization $T=HK$ such that $K$ is cyclic.
Then $(T,H,K)$ is a triple listed in Hypothesis~$\ref{hypo-1}$.
\end{lemma}
		
		\begin{proof}
			Let $\Omega=[T:H]$.
			Then $T$ is a transitive permutation group on $\Omega$, and so is $K$.
			Since $K$ is cyclic, $K$ is regular on $\Omega$.
			Thus $T$ is an almost simple $c$-group of order $n$.
			By \cite[Theorem~1.2(2)]{li2012cyclic}, $(T,n)$ is known, and is one of the following pairs:
			\begin{quote}
				$(\M_{11},11)$, $(\M_{23},23)$, $(\PSL(2,11),11)$, $(\A_n,n)$ with $n$ odd, $(\S_n,n)$, $(\PGL(d,q){:}\l\phi_0\r,\frac{q^d-1}{q-1})$,
				where $\l\phi_0\r$ is a subgroup of a Galois group of the field $\GF(q)$.
			\end{quote}
			We next find out those $(T,n)$ such that $\gcd(|H|,n)=1$.
			First, if $T = \M_{11}$, $\M_{23}$, and $\PSL(2, 11)$, then $\gcd(|H|, n)=1$.
			
		For the pair $(\A_n, n)$, we have $H=\A_{n-1}$ and $n$ is a prime as $\gcd\left(\frac{(n{-}1)!}{2},n\right) = 1$.
			
			Assume that $(T,n)=(\PGL(d,q), \frac{q^{d}-1}{q-1})$.
			Then $H= q^{d-1}{:}\GL(d-1,q)$, and so
			\[\gcd\Bigl(q(q^{d-1}-1)(q^{d-2}-1)\dots(q-1),{q^d-1\over q-1}\Bigr)=1.\]
			It yields that $d$ is a prime.
			Suppose that $d$  is a divisor of $q-1$.
			Then
			\[{q^d-1\over q-1}=q^{d-1}+\dots +q+1=(q^{d-1}-1)+\dots +(q-1)+d\]
			is divisible by $d$.
Hence both $q(q^{d-1}-1)(q^{d-2}-1)\dots(q-1)$ and ${q^d-1\over q-1}$ are divisible by $d$, and so they are not coprime, which is a contradiction.
			
			Conversely, suppose that $d$ is a prime and $\gcd(d,q-1)=1$.
			As $d$ is a prime, we have that $\gcd(q^d-1,q^j-1)=q-1$ for any $1\le j\le d-1$.
			Hence
			\[
			\gcd\Bigl(\frac{q^d-1}{q-1},q^j-1\Bigr)
			=\gcd\Bigl(\frac{q^d-1}{q-1},q-1\Bigr).
			\]
			Noting that $\frac{q^d-1}{q-1}\equiv d\pmod{q-1}$ (see above), it follows that for $1\le j\le d-1$,
			\begin{equation}\label{eq:gcd=1}
			  \gcd\Bigl(\frac{q^d-1}{q-1},q^j-1\Bigr)
			=\gcd(d,q-1)=1.
			\end{equation}
			Therefore, we conclude that $\gcd\left(|T|,n\right)=\gcd\left(\prod_{j<d}(q^j-1),\frac{q^d-1}{q-1}\right)=1$.
		\end{proof}

Now it is ready to prove the first main theorem.
		
\vskip0.1in
{\bf Proof of Theorem~\ref{skewG}:}\ \
Let $G=HK$ be such that $H$ is a Hall subgroup of $G$, and $K$ is cyclic.
To prove the theorem, we assume that $G$ is a minimal counterexample.

Suppose that $H$ is not core-free in $G$.
Let $M$ be the core of $H$ in $G$, so that $M\not=1$.
Then $G/M=(H/M)(KM/M)$ is a Hall factorization by Lemma~\ref{lem:quotient}, and $G/M$ satisfies Theorem~\ref{skewG}.
It yields that $G=HK$ satisfies Theorem~\ref{skewG}, which is not possible.
So $H$ is core-free in $G$.

Let $R$ be the solvable radical of $G$.
Suppose that $R\not=1$.
By Lemma~\ref{subnormal} and Lemma~\ref{lem:solvable}, we obtain
		\[R=K_0{:}\calO_0,\]
where $K_0\leqslant K$ and $\calO_0\leqslant\Aut(K_0)$ is abelian.
By Lemma~\ref{lem:quotient}, $G/R$ satisfies Theorem \ref{skewG}, and we have that $G/R=(T_1\times\dots\times T_r).\calO_1$.
Let $W=R.(T_1\times\cdots\times T_r)$. Then $W\lhd G$. Since $R=K_0{:}\mathcal{O}_0$ and $\gcd(|K_0|,|\mathcal{O}_0|)=1$, we conclude that $K_0\char R$, which yields $K_0\lhd W$. Noting that $K_0$ is cyclic and $\mathcal{O}_0\leq \Aut(K_0)$, we have $W/\C_{W}(K_0)=\mathcal{O}_0$, and $\C_W(K_0)=K_0.(T_1\times\cdots\times T_k)$.
Since $K_0\le K$, we conclude that  $\gcd(|K_0|,|T_i|)=1$ for each $i$. It follows that $\C_W(K_0)=K_0\times T_1\times\cdots\times T_k$.
Therefore, $W=(K_0\times T_1\times\cdots\times T_k).\mathcal{O}_0$.
Thus, we have
\[G=R.(G/R)=(K_0{:}\calO_0).\big((T_1\times\dots\times T_r).\calO_1\big)=(T_1\times\dots\times T_r\times K_0).\calO,\]
where $\calO=\calO_0.\calO_1$, satisfying Theorem~\ref{skewG}.
This contradiction shows that $G$ does not have non-trivial solvable normal subgroups.

Let $N$ be the socle of $G$, the product of all minimal normal subgroups of $G$.
Then
		\[N=T_1\times T_2\times\dots\times T_r,\]
where $T_i$ is nonabelian simple and $r$ is a positive integer.
By Lemma~\ref{subnormal}, $N$ and each $T_i$ have a Hall factorization.
Let $N=H^*K^*$, and $T_i=H_iK_i$, where $1\leqslant i\leqslant r$.
By Lemma~\ref{2.6}, the tuple $(T_i,H_i,K_i)$ lies in Table~\ref{tab:1} in Hypothesis~\ref{hypo-1}, with $e(T_i)=|K_i|$.
Moreover, the cyclic factor $K^*$ of $N$ is equal to
		\[K^*=K_1\times K_2\times\dots\times K_r.\]
		It follows that $|K_1|,|K_2|,\dots,|K_r|$ are pairwise coprime, so $e(T_1),\dots, e(T_r)$ are pairwise coprime.
		In particular, $T_1,T_2,\dots,T_r$ are pairwise nonisomoprhic.
        We have $G/N=\calO\leqslant\Out(N)=\Out(T_1)\times\dots\times\Out(T_r)$, since $\C_G(N)=1$. It yields $\calO\le H$ and $K\le N$.
Thus, we have $G=(T_1\times\dots\times T_r).\calO$, and $G$ satisfies Theorem \ref{skewG}. This contradiction shows that $G$ always satisfies Theorem~\ref{skewG}, completing the proof.
\qed

The following proposition shows that the number $r$ of simple factors in $G$ can be arbitrarily large.
		
\begin{proposition}\label{any r}
Let $d$, $p$ be two primes with $d>p$.
Let $d_i$ with $1\leqslant i\leqslant r$ be distinct primes such that $d<d_1<\dots<d_r<d^2$, and set $q_i=p^{d_i}$.
Then
	\[\gcd(d_i,\ q_i-1)=1,\ \mbox{and}\ \gcd\left({q_i^{d_i}-1\over q_i-1},\ q_j^k-1\right)=1,\]
for any $i,j\in\{1,2,\dots,r\}$ and positive integer \textcolor{red}{}{$k\le d_j$}.
Moreover, $r\to \infty$ as $d\to \infty$.
\end{proposition}

\begin{proof}
Since $d_i$ and $p$ are distinct primes, $d_i$ divides $p^{d_i-1}-1$ by Fermat Little Theorem.
If $d_i$ divides $p^{d_i}-1$, then $d_i$ divides $\gcd(p^{d_i-1}-1,p^{d_i}-1)=p-1$, which contradicts the assumption $d_i>p$.
Hence $d_i$ does not divide $p^{d_i}-1$, and so $\gcd(d_i,p^{d_i}-1)=1$, the first equality is proved.

Next we prove the second equality.
Since $d_i$ is coprime to $q_i-1$, and $q_i-1$ divides $q_i^e-1$ for each positive integer $e$, it yields that \[\frac{q_i^{d_i}-1}{q_i-1}=q_i^{d_i-1}+\dots+q_i+1=(q_i^{d_i-1}-1)+\dots+(q_i-1)+d_i\]
is coprime to $q_i-1$, and hence $\gcd\bigl(\frac{q_i^{d_i}-1}{q_i-1},\,q_i-1\bigr)=1$.

For $i=j$, we have $\gcd\left({q_i^{d_i}-1\over q_i-1},\ q_i^k-1\right)=1$ by \eqref{eq:gcd=1}.
Thus we assume that $i\neq j$.
By Euclidean algorithm, we deduce
			\[\gcd(q_i^{d_i}-1,q_j^k-1)=\gcd(p^{d_i^2}-1,p^{kd_j}-1)=p^{\gcd(d_i^2,d_j k)}-1=p^{\gcd(d_i^2,k)}-1.\]
We claim that
			\[\mbox{$\gcd(d_i^2,k)=\gcd(d_i,k)$, for any $1\leqslant k\le d_j$.}\]
If $i>j$, then $k\leqslant d_j<d_i$, and $\gcd(d_i,k)=\gcd(d_i^2,k)=1$ as $d_i$ is a prime.
			In the case where $i<j$, we have $d_i<d_j<d^2<d_i^2$.
			It yields that $\gcd(d_i,k)=\gcd(d_i^2,k)=d_i$ if $d_i\mid k$, and
			$\gcd(d_i,k)=\gcd(d_i^2,k)=1$ if $d_i$ does not divide $k$.
			The Claim is justified.
			Thus we conclude that
			\[\gcd(q_i^{d_i}-1,q_j^k-1)=p^{\gcd(d_i^2,k)}-1=p^{\gcd(d_i,k)}-1\]
			divides $p^{d_i}-1=q_i-1$, and so
			\[
			\gcd\Bigl(\frac{q_i^{d_i}-1}{q_i-1},\,q_j^k-1\Bigr)=1
			\quad\text{for all }i\neq j, \,1\le k\le d_j.
			\]
			
			Finally, letting $\pi(n)$ be the number of primes which are at most $n$,
			by the Prime Number Theorem, we have
			\[
			r = \pi(d^2)-\pi(d) \sim \frac{d^2}{\ln(d^2)} - \frac{d}{\ln d}
			= \frac{d^2-2d}{2\ln d}.
			\]
			The number $r$ of prime numbers lying between $d$ and $d^2$ approaches $\infty$ if $d$ goes to $\infty$.
			This completes the proof of the proposition.
		\end{proof}

Now we can present an explicit family of examples with arbitrarily large $r$.

\begin{example}\label{ex:r-nobound}
{\rm For primes $p<d<d_1<\dots<d_r<d^2$, let $T_i=\PSL(d_i,p^{d_i})$ with $1\leqslant i\leqslant r$.
Then $\gcd(|T_i|,e(T_j))=1$ for any $i\not=j$, and so the group
		$$\PSL(d_1,p^{d_1})\times\PSL(d_2,p^{d_2})\times\dots\times\PSL(d_r,p^{d_r})$$
has a  factorization with a cyclic factor of order $\prod_{1\leqslant i\leqslant r}{p^{d_i^2}-1\over p^{d_i}-1}$ \text{as its Hall subgroup}.}
\end{example}

{\bf Proof of Corollary~\ref{cor:r-nobound}:}\
By Example~\ref{ex:r-nobound}, there is no upper bound for the number $r$ of the direct factors $T_i$'s.
\qed

{\bf Proof of Theorem~\ref{thm:skew-G}:}\
Let $H$ be a finite group which has a skew-morphism $\rho$.
Then there exists a group $G=H\l\rho\r$ such that $\l\rho\r$ is core-free in $G$.
Hence the triple $(G,H,\l\rho\r)$ is a triple $(G,H,K)$ described in Theorem~\ref{skewG}, so that
\[H=N.(H_1\times \dots\times H_r).\calO,\]
where each $H_i<T_i$ with $(H_i,T_i)$ being a pair $(H,T)$ given in Hypothesis~\ref{hypo-1}.
Let
\[\ell_i=e(T_i),\ \mbox{where $1\leqslant i\leqslant r$}.\]

Without loss of generality, we may assume that, for some $s$ with $1\leqslant s\leqslant r$,
\begin{itemize}
\item $H_i\in\{\A_{p-1},\S_{p-1},\A_5,\M_{10},\M_{22}\}$ for $i\leqslant s$, and
\item $H_i=\AGL(d_i,q_i)$ or $\AGL(d_i,q_i).\l\phi_i\r$ for $i>s$.
\end{itemize}

Now we determine $L=\prod_{i\leqslant s}H_i$.
We claim that $\A_p$ can appear at most once among the $T_i$'s with $i\leqslant s$.
Suppose that $T_1=\A_{p_1}$ and $T_2=\A_{p_2}$ with $p_1\leqslant p_2$.
Then $\ell_1=p_1$ and $\ell_2=p_2$.
Clearly, $p_1=\ell_1\not=\ell_2=p_2$, so $p_1<p_2$.
Then $p_1\mid |\A_{p_2-1}|=\frac{(p_2-1)!}{2}$, which is a contradiction since $\gcd(|\A_{p_2}|,p_1)=1$.
Similar arguments show that none of $\PSL(2,11)$, $\M_{11}$ or $\M_{23}$ can appear twice.

Suppose $s\geqslant2$.
Then $H_1\times H_2$ has a Hall skew-morphism such that $H_i<T_i$ and
\[\{T_1,T_2\}\subset\{\A_p,\S_p,\PSL(2,11),\M_{11},\M_{23}\},\ \mbox{where $p$ is a prime}.\]

Assume first that $T_1=\A_p$.
Then $T_2\in \{\PSL(2,11),\M_{11},\M_{23}\}$, and so $T_1\times T_2=(H_1\times H_2)\l\rho\r$, where
\[\mbox{$(H_1\times H_2)\l\rho\r=(\A_{p-1}\times\A_5)\ZZ_{11p}$, $(\A_{p-1}\times\M_{10})\ZZ_{11p}$, or $(\A_{p-1}\times\M_{22})\ZZ_{23p}$.}\]
It yields that $T_1\times T_2=\A_7\times\PSL(2,11)$, $\A_7\times\M_{11}$, or $\A_{13}\times\M_{23}$, $\A_{17}\times\M_{23}$ or $\A_{19}\times\M_{23}$, which are listed in the theorem.

If $T_1=\PSL(2,11)$, then $T_2=\M_{11}$ or $\M_{23}$, which is not possible.

If $T_1=\M_{11}$, then $T_2=\M_{23}$, which is not possible.
This completes the proof.
\qed

		

\noindent {\bf Proof of Corollary~\ref{cor:solvable}}.

Inspecting the candidates $(T,H,K)$ given in Hypothesis~\ref{hypo-1} with $H$ being solvable, we conclude that $(T,H,K)$ is in the following table.
\[
\begin{array}{ccccc}
\hline
T & e(T) & H & K & \text{remark} \\ \hline

\PSL(2,q) & q+1 & \AGL(1,q) & \ZZ_{q+1} & q=2^f \\

\PSL(3,2) & 7 & \S_4 & \ZZ_7 & \\

\PSL(3,3) & 13 & \AGL(2,3) & \ZZ_{13} & \\
\hline
\end{array}
\]
Suppose that $T_1=\PSL(2,2^e)=H_1K_1$ and $T_2=\PSL(2,2^f)=H_2K_2$, where $e<f$.
Then $K_1\times K_2$ is cyclic, so $\gcd(2^e+1,2^f+1)=1$.
As $\gcd(|H_1||H_2|,|K_1||K_2|)=1$, we have
			\[\gcd(2^e-1,2^f+1)=1,\ \gcd(2^e+1,2^f-1)=1.\]
Therefore, we obtain
\[\begin{aligned}
				2^{2\cdot\gcd(e,f)} - 1
				&= \gcd\bigl(2^{2e} - 1,\, 2^{2f} - 1\bigr) \\
				&\leq \gcd(2^{e} + 1,\, 2^{f} + 1)
				\cdot \gcd(2^{e} + 1,\, 2^{f} - 1) \\
				&\quad\ \cdot \gcd(2^{e} - 1,\, 2^{f} + 1)
				\cdot \gcd(2^{e} - 1,\, 2^{f} - 1) \\
				&= 2^{\gcd(e,f)} - 1,
			\end{aligned}
			\]
which is a contradiction.
That is to say, among the direct factor of $N=T_1\times\dots\times T_r$, at most one $T_i$ is of the form $\PSL(2,2^f)$.
Thus $N=T_1\times\dots\times T_r$ is such that $r\leqslant3$.

If $r=1$, then obviously $N=T$ is as in the above table.
		
		Assume that $r\geqslant2$.
		Then $N>\PSL(3,2)$ or $\PSL(3,3)$.
		Assume further that $N\not=\PSL(3,2)\times\PSL(3,3)$.
Then $N=\PSL(3,2)\times\PSL(2,2^f)$, $\PSL(3,3)\times\PSL(2,2^f)$, or $\PSL(3,2)\times\PSL(3,3)\times\PSL(2,2^f)$.
		
		Suppose that $T_1=\PSL(3,2)=H_1K_1$ and $T_2=\PSL(2,2^f)=H_2K_2$.
		Then $\gcd(|T_1|,2^f+1)=1$, yielding that $f$ is even.
		If $f$ is divisible by 6, then $2^f-1$ is divisible by $2^6-1$ so by 7;
		however, $e(T_1)=2^2+2+1=7$ should be coprime to $|T_2|$, which is not possible.
		So conclude that $f\equiv 2$ or 4 $(\mod 6)$.
Similarly, if $N=\PSL(3,2)\times\PSL(3,3)\times\PSL(2,2^f)$, then $f\equiv 2$ or 4 $(\mod 6)$.
		
		Suppose that $T_1=\PSL(3,3)$ and $T_2=\PSL(2,2^f)=H_2K_2$.
		Then $e(T_1)=3^2+3+1=13$ divides $2^6+1$, yielding that $f$ is not divisible by 6.
		So we conclude that $q\equiv 2$ or $4$ $(\mod 6)$.
		\qed

\def\PGaL{{\rm P\Gamma L}}
\def\PSigL{{\rm P\Sigma L}}

\section{Hall Cayley maps}\label{sec:3}
		
In this section, we prove Theorem~\ref{thm:maps} by a series of lemmas.
		
Let $\calM=(V,E,F)$ be a Cayley map of $H$ which is $G$-vertex-rotary, where $G\leqslant\Aut\calM$.
Then there exists a rotary pair $(\rho,z)$ for $G$, so that
\[G=\l\rho,z\r.\]
Let $\a$ be the vertex corresponding to the identity of $H$.
Then $G_\a=\l\rho\r$ is regular on the edge set $E(\a)$, and
\[G=HG_\a,\ \mbox{and $H\cap G_\a=1$}.\]
Assume that $\calM$ is a Hall Cayley map of $H$.
Then $H$ is a Hall subgroup of $G$.

We first show that $\calM$ is a core-free Hall Cayley map if and only if $H$ is core-free in $G$. The sufficiency follows since $\core_{\Aut\calM}(H)\leq \core_G(H)=1$. 
Note that $\Aut\calM=G{:}2$ or $G$, so we assume that $\Aut\calM=G{:}\l x\r$, where $|x|=2$. Denote $N=\core_G(H)$, then $1=\core_{\Aut\calM}(H)=N\cap N^x$. Since $N^x\lhd G$,  $N^xH$ is a Hall subgroup of $G$, which implies that $N^xH=H$. Thus, $N^x\leq H$. Since $N=\core_G(H)$ and $N^x\lhd G$, we conclude that $N^x=N$. Therefore, $N=\core_{(G{:}\l x\r)}(H)=1$.

So, in order to study core-free Hall Cayley map, we
assume that both $H$ and $K$ are core-free in $G$.
		Then by Theorem~\ref{skewG}, we obtain that $G=(T_1\times \dots\times T_r).\calO$,
		where $\gcd(|T_i|,e(T_j))=1$ for any $i\not=j$, and
		$T_i=H_iK_i$ is a simple group satisfying Hypothesis~$\ref{hypo-1}$ such that
		\begin{itemize}
			\item[\rm(i)] $H=(H_1\times\dots\times H_r).\calO$, with $\calO\leqslant\Out(T_1)\times\dots\times\Out(T_r)$,
			\item[\rm(ii)] $K=K_1\times\dots\times K_r$.
		\end{itemize}
		
The following lemma determines $\calO$.
		
\begin{lemma}\label{lem:O=<2}
With notation given above, $\calO=1$ or $\ZZ_2$, we have that $G=T_1\times \dots\times T_r$ or $(T_1\times \dots\times T_r).2$, and $\rho\in T_1\times \dots\times T_r$.
\end{lemma}
		
		\begin{proof}
			Let $(\a,e,f)$ be a flag of the map $\calM$.
			Since $G$ is transitive on the vertex set $V$, we may assume that $K=G_\a$.
			Let $M=\soc(G)=T_1\times \dots\times T_r$.
			Then we have $G_\a<M$.
Thus $M\geqslant \l G_\b\mid \b\in V\r$, and hence $M$ is transitive on the edge set $E$.
Since $G$ is arc-regular on $\calM$, it follows that either $M$ is arc-transitive on $\calM$ and $M=G$, or $M$ is edge-regular on $\calM$ and $G=M.2$.
We therefore conclude that $\calO=G/M$ equals 1 or $\ZZ_2$.
\end{proof}
		
Next,  we shall completely determine almost simple groups $G$.
We first construct rotary pairs $(\rho,z)$ for each possible almost simple group $G$ in Hypothesis~\ref{hypo-1}, so that $G$ has a rotary map $\RotaMap(T, \rho,z)$ and a  bi-rotary map $\BiRoMap(T,\rho,z)$.

\begin{example}\label{ex:S(p)}
\rm Let $G=\S_p=HK$ with $H=\S_{p-1}$ and $K=\l \rho\r=\ZZ_p$, acting on $\Omega=\{1,2,\dots,p\}$.
Write $\rho=(12\dots p)\in G$.
Then the involution $z=(12)\in G$ is such that $(\rho,z)$ is a rotary pair for $G=\A_p{:}\l z\r$.
\qed
\end{example}
        		
\begin{example}\label{ex:Alt(p)}
\rm{
Let $T=\A_p=HK$ with $H=\A_{p-1}$ and $K=\l\rho\r=\ZZ_p$, acting on $\Omega=\{1,2,\dots,p\}$.
Write $\rho=(12\dots p)\in T$.
Let $z=(12)(34)\in T$.
Then
				\[zz^\rho=(13542)\]
is a 5-cycle, and by \cite[Theorem\,3.3E]{Dixon-book}, either $\l \rho,z\r=T$, or $p=7$.
For the case where $p=7$, it is easy to see that $\l \rho,z\r=T$ since $\l \rho,z\r$ contains elements of order 7 and order 5.
Therefore, $(\rho,z)$ is a rotary pair for $T$, and thus $H=\A_{p-1}$ has two Hall Cayley maps of $H$.
				
Moreover, $\rho z=(245\dots p)$, which is of order $p-2$ (and fixes the points 1 and 3), and $\l z,z^\rho\r=\l(12)(34),(23)(45)\r=\l(13542)\r{:}\l(12)(34)\r=\D_{10}$.
Thus we have
\begin{itemize}
\item $\RotaMap(T,\rho,z)$ has face stabilizer $\l\rho z\r=\ZZ_{p-2}$, and
\item $\BiRoMap(T,\rho,z)$ has face stabilizer $\l z,z^\rho\r=\D_{10}$. \qed
\end{itemize}
}
\end{example}

\begin{example}\label{ex:L(2,11)}
{\rm
Let $T=\PSL(2,11)=HK$, where $H=\A_5$ and $K=\ZZ_{11}$.
Pick an element $\rho\in K$, of order 11.
Then, for any involution $z\in T$, the pair $(\rho,z)$ is a rotary pair.
}
\qed\end{example}
		
\begin{example}\label{ex:M11}
{\rm
Let $T=\M_{11}=HK$, where $H=\M_{10}=\A_6.2$ and $K=\ZZ_{11}$.
Pick an element $\rho\in K$, of order 11.
Then, for any involution $z\in T$, either $\l\rho,z\r=\M_{11}$ or $\l\rho,z\r=\PSL(2,11)$, see \cite{Atlas}.
By MAGMA~\cite{MAGMA}, the maximal subgroup of $\M_{11}$ which contains $\rho$ is unique and is isomorphic to $\PSL(2,11)$.
				
Moreover, $\M_{11}$ contains ${2^4\cdot3^2\cdot5\cdot11\over 48}=3\cdot5\cdot11$ involutions, and $\PSL(2,11)$ contains ${2^2\cdot3\cdot5\cdot11\over 12}=5\cdot11$ involutions.
Thus there exist involutions $z\in T$ such that $\l\rho,z\r=T$, and so the pair $(\rho,z)$ is a rotary pair.
}
\qed\end{example}

		\begin{example}\label{ex:M23}
			{\rm
				Let $T=\M_{23}=HK$, where $H=\M_{22}$ and $K=\ZZ_{23}$.
				Pick an element $\rho\in K$, of order 23.
				Then, for any involution $z\in T$, we have $\l\rho,z\r=T$, see \cite{Atlas}.
			}
		\qed\end{example}

\begin{example}\label{ex:PSL}
{\rm
Let $T=\PSL(d,q)=HK$, where $d$ is a prime,  $\gcd(d,q-1)=1$, $H=\AGL(d-1,q)$ and $K=\ZZ_{q^{d}-1\over q-1}$.
Pick an element $\rho\in K$, of order ${q^{d}-1\over q-1}$.
In the case where $d>2$, each involution $z\in T$ is such that $\l\rho,z\r=T$, and thus $(\rho,z)$ is a rotary pair for $T$.
In the case where $d=2$, each involution $z\notin \N_T(K)\cong\D_{2(q+1)}$ is such that $\l\rho,z\r=T$, and so $(\rho,z)$ is a rotary pair for $G=\PSL(2,q)$. Note that there are $q^2-1$ involutions in $\PSL_2(q)$ and $q+1$ involutions in $\D_{2(q+1)}$. Therefore, such $z$ exists.
}
\qed\end{example}

\begin{example}\label{ex:PSL.2}
{\rm
Let $G=\PSL(d,q){:}\l\phi\r=HK$, where $d$ is a prime, $\gcd(d,q-1)=1$, $\phi$ is a field automorphism of order $2$, $H=\AGL(d-1,q){:}\l\phi\r$ and $K=\ZZ_{q^{d}-1\over q-1}$.
Choose an element $\rho\in K$ of order ${q^{d}-1\over q-1}$,
and let $q=q_0^2$.
Pick $x$ to be an involution of $\PSL(d,q_0)$ such that $x\notin \N_{\PSL(d,q)}(K)\cong K{:}\ZZ_d$.
Note that $\PSL(d,q_0)$ is centralized by $\phi$.
Then  $x\phi$ is an involution in $G$, and $\l\rho,x\phi\r=G$. Thus $(\rho,z)$ is a rotary pair for $G=\PSL(d,q){:}\l z\r$ with $z=x\phi$.
When $d>2$, take $x$ to be any involution of $\PSL(d,q_0)$.
When $d=2$, there exists at most one involution of $\PSL(2,q_0)$ contained in $\N_{\PSL(2,q)}(K)\cong \D_{2(q+1)}$, so that such $x$ exists. If not, assuming  $x_1,x_2$ are such involutions, we have $|\l x_1,x_2\r|\mid \gcd(2(q+1),q_0(q-1))$.
Note that $\gcd(q-1,q+1)=1$ since $q$ is even. We have $\l x_1,x_2\r=\ZZ_2$, a contradiction.
}
\qed\end{example}

The next lemma completely determines almost simple groups $G$.

\begin{lemma}\label{lem:O=2}
The group $G$ is an almost simple group if and only if $G$ is a simple group listed in Hypothesis~$\ref{hypo-1}$, or $G=T{:}\l z\r$ such that either
\begin{itemize}
\item[\rm(i)] $T=\A_p$ and $z$ is an odd permutation in $\S_p$, or

\item[\rm(ii)]  $T=\PSL(d,q)$ and $z=x\phi$, where $\phi$ is a field automorphism of order $2$, and $x^\phi=x^{-1}$.
\end{itemize}
\end{lemma}

\begin{proof}
The sufficiency has been confirmed by the above examples.

Next we verify the necessity, so that assume that $G$ is an almost simple group.
Then by Lemma~\ref{2.6}, $G$ is one of the almost simple groups listed in Hypothesis~\ref{hypo-1}.
Assume further that $G$ is not simple, and $G\not=\S_p$ with $p$ prime.
Then
\[G=\PSL(d,q){:}\l \phi\r,\]
where $\phi$ is a field automorphism of order $2$.
In this case, $\rho \in T$ and $z\notin T$, where $T=\PSL(d,q)$.  Since $G=T{:}\l\phi\r=T{:}\ZZ_2$, we have $G=T{:}\l z\r$.
It follows that $z=x\phi$ with $x\in T$, and $1=z^2=x\phi x\phi$, so  $\phi^{-1}x\phi=x^{-1}$.
\end{proof}

The following lemma classifies the groups $G$ in the general case.

\begin{lemma}\label{lem:exist}
Letting $T_0{:}\l z_0\r=1$, there exists $s$ with $0\leqslant s\leqslant r$ such that
\[G=\left((T_0\times\dots\times T_s){:}\l (z_0,\dots, z_s)\r\right)\times T_{s+1}\times\dots\times T_r,\]
where $(T_i,z_i)$ is a pair given in Lemma \ref{lem:O=2}.
Further, let $(\rho_i,z_i)$ be a rotary pair of $T_i{:}\l z_i\r$ for $i\leqslant s$ and a rotary pair of $T_i$ for $i>s$, and let $\rho=(\rho_1,\dots,\rho_r)$ and $z=(z_1,\dots, z_r)$.
Then $(\rho,z)$ is a rotary pair of $G$.
\end{lemma}
\begin{proof}
Assume that $G\le (T_1\times\dots\times T_r).\ZZ_2$. If $G=T_1\times \cdots T_r$, then take $s=0$. Now, assume that $G=(T_1\times \cdots\times T_r).\ZZ_2$.
Then $T_i\C_G(T_i)=T_1\times\dots\times T_r$ for some $i$ with $1\leqslant i\leqslant r$.
Without loss of generality, we may assume that $0<s\leqslant r$ is the largest value such that $T_i\C_G(T_i)=T_1\times\dots\times T_r$ for $1\leqslant i\leqslant s$.
Then
\[G=\left((T_1\times\dots\times T_s).\ZZ_2\right)\times(T_{s+1}\times\dots\times T_r),\]
and $G/\C_G(T_i)\cong T_i.\ZZ_2$.
By Lemmas~\ref{lem:quotient} and \ref{lem:O=2}, we conclude that $G/\C_G(T_i)\cong T_i{:}\ZZ_2=T_i{:}\l z_i\r$, with $|z_i|=2$.
It yields that
$$(T_1\times\dots\times T_s).\ZZ_2=(T_1\times\dots\times T_s){:}\l(z_1,\dots,z_s)\r.$$

Next, we show that $(\rho,z)$ is a rotary pair of $G$. If $s=0$, then $G=T_1\times \cdots\times T_r$. Since $T_i\not\cong T_j$ for $i\neq j$, we have $\l \rho,z\r=G$.
If $s>0$, then $G=(T_1\times\cdots\times T_r).\ZZ_2$. Since $T_i\not\cong T_j$ for $i\neq j$,  
we conclude that $$T_1\times\cdots\times T_r\leq \l \rho,z\r\leq G.$$ 
Note that $z\notin T_1\times\cdots\times T_s$. Therefore, $(\rho,z)$ is a rotary pair of $G$.
\end{proof}

As mentioned in the Introduction, a group $G$ with a rotary pair $(\rho,z)$ determines two different arc-transitive maps:
a rotary map $\RotaMap(G,\rho,z)$ and a bi-rotary map $\BiRoMap(G,\rho,z)$, both of which have underlying graph $\Ga=\Cos(G,\l\rho\r,\l\rho\r z\l\rho\r)$.
In the rest of this section, we decompose the graph $\Ga$ as direct product or bi-direct product of graphs admitting almost simple groups.

Let $\Ga$ and $\Sig$ be graphs with vertex sets $U$ and $V$.
Then the {\it direct product} $\Ga\times\Sig$ is the graph with vertex set $U\times V$ such that
\[\mbox{$(u_1,v_i)\sim(u_{2},v_{2})\Longleftrightarrow$ $u_1\sim u_{2}$ and $v_1\sim v_{2}$},\]
for any $u_i\in U$ and $v_i\in V$, $i=1,2$.

Observe that, if $\Ga$ and $\Sig$ are bi-partite graphs, $\Ga\times\Sig$ is not connected and has two connected components.
Let $U_1$, $V_1$ be the two parts of vertices sets of $\Ga$, and let $U_2$, $V_2$ be the two parts of vertices sets of $\Sig$.
Then the {\it bi-direct product} $\Ga\times_{\rm{bi}} \Sig$ is a bi-partite graph with two parts of
vertex set $U_1\times U_2$ and $V_1\times V_2$ such that
\[\mbox{$(u_1,u_2)\sim (v_1,v_2)\Longleftrightarrow$ $u_1\sim v_1$ and $u_2\sim v_2$,}\]
for any $u_i\in U_i$ and $v_i\in V_i$, $i=1,2$, see \cite{LiMa}.		

\begin{lemma}\label{lem:productmap}
Let $G=\left((T_0\times \dots\times T_s){:}\l (z_0,\dots, z_s)\r\right)\times T_{s+1}\times\dots\times T_r$, and let  $\rho=(\rho_1,\dots,\rho_r)$ and $z=(z_1,\dots, z_r)$, defined as in Lemma~$\ref{lem:exist}$.
Let $\Ga_i$ be the underlying graph of $\RotaMap(T_i{:}\l z_i\r,\rho_i,z_i)$ for $i\leq s$, or of  $\RotaMap(T_i,\rho_i,z_i)$ for $i>s$.
Then the underlying graph of $\RotaMap(G,\rho,z)$ is such that
\[\Cos(G,\l \rho\r,\l\rho\r z\l\rho\r)=(\Ga_1\times_{\rm{bi}}\Ga_2\times_{\rm{bi}}\cdots\times_{\rm{bi}}\Ga_s)\times \Ga_{s+1}\times\cdots\times\Ga_r.\]
\end{lemma}
\begin{proof}
{\bf (1).}\ First, assume  that $G=T_1\times T_2$.
Then $\rho=(\rho_1,\rho_2)$, and $G_\a=\l\rho\r=\l\rho_1\r\times\l\rho_2\r$ as $\gcd(|\rho_1|,|\rho_2|)=1$.
Hence
\[V=[G:G_\a]=[(G_1\times G_2):(\l\rho_1\r\times \l\rho_2\r)] =[G_1:\l \rho_1\r]\times[G_2:\l\rho_2\r].\]
For any two vertices $\l\rho\r (s_1,s_2)$, $\l \rho\r (t_1,t_2)$ in $V(\Ga)$, we have
\[\begin{array}{rcl}
\l(\rho_1,\rho_2)\r (s_1,s_2)\sim\l (\rho_1,\rho_2)\r (t_1,t_2)&\Longleftrightarrow&
(s_1,s_2)(t_1^{-1},t_2^{-1})\in \l(\rho_1,\rho_2)\r z_1z_2\l(\rho_1,\rho_2)\r \\
&\Longleftrightarrow& s_1t_1^{-1}\in \l\rho_1\r z_1\l\rho_1\r, \ s_2t_2^{-1}\in \l\rho_2\r z_2\l\rho_2\r, \\
&\Longleftrightarrow& \l\rho_1\r s_1\sim\l \rho_1\r t_1, \ \l\rho_2\r s_2\sim\l \rho_2\r t_2\\
\end{array}\]
By definition, we conclude that $\Ga=\Ga_1\times\Ga_2$.

Suppose now that $G=T_1\times\dots\times T_r$, with $r\geqslant3$.
Let $X=T_1\times\dots\times T_{r-1}$.
Then $G=X\times T_r$.
Let $\rho'=(\rho_1,\dots,\rho_{r-1})$, and $z'=(z_1,\dots,z_{r-1})$.
Then $(\rho',z')$ is a rotary pair for $X$, and defines a graph $\Sig=\Cos(X,\l\rho'\r,\l z'\r)$.
By the previous paragraph, we conclude that $\Ga=\Sig\times\Ga_r$.
By induction, $\Sig=\Ga_1\times\dots\times \Ga_{r-1}$.
It follows that $\Ga=\Sig\times\Ga_r=(\Ga_1\times\dots\times \Ga_{r-1})\times\Ga_r=\Ga_1\times\dots\times \Ga_{r-1}\times\Ga_r$.

{\bf (2).}\ Next, assume that $G=(T_1\times T_2){:}\l(z_1,z_2)\r$, where $|z_1|=|z_2|=2$.
Let $G_i=T_i{:}\l z_i\r$, and $\Ga_i=\Cos(T{:}\l z_i\r,\l\rho_i\r, \l\rho_i\r z_i\l\rho_i\r)$, where $i=1$ or 2.
Let $\Ga=\Cos(G{:}\l z\r,\l\rho\r, \l\rho\r z\l\rho\r)$, where $\rho=(\rho_1,\rho_2)$ and $z=(z_1,z_2)$.
Then $\Ga_i$ and $\Ga$ are bipartite graphs.
Let 
$$\begin{array}{l}
U=\{\l \rho\r(s_1,s_2)\mid (s_1,s_2)\in T_1\times T_2\}, \text{ and } V=\{\l\rho\r z(t_1,t_2) \mid (t_1,t_2)\in T_1\times T_2\},\\
U_i=\{\l \rho_i\r s_i\mid s_i\in T_i\}, \text{ and } V_i=\{\l\rho_i\r z_i t_i \mid t_i\in T_2\}, \mbox{where $i=1$ or 2}.
\end{array}$$
		
Notice that a vertex in $V$ has the form $\l \rho\r z(t_1,t_2)=\l \rho\r (z_1t_1,z_2t_2)$.
Then, for any vertices $\l\rho\r (s_1,s_2)\in U$ and $\l \rho\r (z_1t_1,z_2t_2)\in V$, we have that
\[\begin{array}{rcl}
\l\rho\r (s_1,s_2)\sim\l \rho\r (z_1t_1,z_2t_2)&\Longleftrightarrow&
(s_1,s_2)(z_1t_1,z_2t_2)^{-1}\in \l(\rho_1,\rho_2)\r (z_1,z_2)\l(\rho_1,\rho_2)\r \\
&\Longleftrightarrow& s_1t_1^{-1}z_1\in \l\rho_1\r z_1\l\rho_1\r, \mbox{and}\ s_2t_2^{-1}z_2\in \l\rho_2\r z_2\l\rho_2\r, \\
&\Longleftrightarrow& \l\rho_1\r s_1\sim\l \rho_1\r z_1t_1, \mbox{and}\ \l\rho_2\r s_2\sim\l \rho_2\r z_2t_2.\\
\end{array}\]
By definition, we conclude that $\Ga=\Ga_1\times_{\rm{bi}}\Ga_2$.

Now let $G=(T_1\times\dots\times T_s){:}\l (z_1,\dots, z_s)\r=(X\times T_r){:}\l(z',z_r)\r$.
Let $\Ga'=\Cos(X{:}\l z'\r,\l\rho'\r, \l\rho'\r z'\l\rho'\r)$, where $\rho'=(\rho_1,\dots,\rho_{r-1})$.
Arguing as in the previous paragraph shows that $\Ga=\Ga'\times_{\rm bi}\Ga_r$.
By induction, we may assume that $\Ga'=\Ga_1\times_{\rm bi}\dots\times_{\rm bi}\Ga_{r-1}$.
Thus $\Ga=\Ga'\times_{\rm bi}\Ga_r=\Ga_1\times_{\rm bi}\dots\times_{\rm bi}\Ga_{r-1}\times_{\rm bi}\Ga_r$.

{\bf (3).}\ Assume that $G=\left((T_1\times\dots\times T_s){:}\l (z_1,\dots, z_s)\r\right)\times (T_{s+1}\times\dots\times T_r)$, where $1<s<r$.
Let $\rho'=(\rho_1,\dots,\rho_s)$, $\rho''=(\rho_{s+1},\dots,\rho_r)$, and let $z'=(z_1,\dots, z_s)$, and $z''=(z_{s+1},\dots,z_r)$.
Let $X=T_1\times\dots\times T_s$ and $Y=T_{s+1}\times\dots\times T_r$.
Then $G=(X{:}\l z'\r)\times Y$.
Let $\Ga'=\Cos((X{:}\l z'\r,\l\rho'\r z'\l\rho'\r)$, and let $\Ga''=\Cos((Y,\l\rho''\r z''\l\rho''\r)$.
Arguing as in the first paragraph of the proof shows that $\Ga=\Cos(G,\l\rho\r z\l\rho\r)=\Ga'\times\Ga''$.
Further, $\Ga'=\Ga_1\times_{\rm{bi}}\cdots\times_{\rm{bi}}\Ga_s$ by (1), and $\Ga''=\Ga_{s+1}\times\cdots\times\Ga_r$ by (2).
So
$\Ga=\Ga'\times\Ga''=(\Ga_1\times_{\rm{bi}}\Ga_2\times_{\rm{bi}}\cdots\times_{\rm{bi}}\Ga_s)\times (\Ga_{s+1}\times\cdots\times\Ga_r)$.
\end{proof}

\vskip0.1in
Finally, combining Lemma~\ref{lem:exist} and Lemma~\ref{lem:productmap}, {\bf the proof of Theorem~\ref{thm:maps}} is completed.
\qed

		

\begin{thebibliography}{20}
			
					\bibitem{BCMG2022}
             M. Bachrat\'y, M.~D.~E. Conder and G. Verret, Skew product groups for monolithic groups, Algebr. Comb. {\bf 5} (2022), no.~5, 785--802.


            
			\bibitem{BachratyJajcay2017}
			M.~Bachrat\'y and R.~Jajcay,
            Classification of coset-preserving skew-morphisms of finite cyclic groups, Australas. J. Combin. {\bf 67} (2017), 259--280.

            \bibitem{MAGMA}
	          W.~Bosma, J.~Cannon, and C.~Playoust,
	        The magma algebra system. {I}. {T}he user language,
	        J. Symbolic Comput. {\bf 24} (1997), 235--265.

			\bibitem{ChenDuLi22}
			J.~Chen, S.F. Du, and C.H. Li,
			Skew-morphisms of nonabelian characteristically simple groups, J. Combin. Theory Ser. A {\bf 185} (2022), Paper No. 105539.
			\bibitem{ConderJajcayTucker2016}
			M.~Conder, R.~Jajcay, and T.W. Tucker,
			\newblock Cyclic complements and skew morphisms of groups, J. Algebra {\bf 453} (2016), 68--100.
			
			\bibitem{ConderTucker2014}
			M.~Conder and T.W. Tucker,
			\newblock  Regular Cayley maps for cyclic groups, Trans. Amer. Math. Soc. {\bf 366} (2014), no.~7, 3585--3609.
			
			\bibitem{Atlas}
			J.H. Conway, R.T. Curtis, S.P. Norton, R.A. Parker, and R.A. Wilson,
			\newblock  Atlas of Finite Groups,
			\newblock Oxford University Press, Eynsham, 1985.

\bibitem{Hall-maps}
W.D. Di, Z. Guo, and C.H. Li, 
\newblock Vertex-rotary Hall Cayley maps, in preparation.

            \bibitem{Dixon-book}
			  J.D. Dixon and B. Mortimer,
			 Permutation Groups,
			\newblock Springer-Verlag, Berlin, New York, 1996.

			\bibitem{DuHu2019}
			S.F. Du and K.~Hu,
			\newblock Skew-morphisms of cyclic 2-groups, J. Group Theory {\bf 22} (2019), no.~4, 617--635.
			
			\bibitem{DuLuoYuZhang24}
			S.F. Du, W.~Luo, H.~Yu, and J.Y. Zhang,
			\newblock Skew-morphisms of elementary abelian $p$-groups, J. Group Theory {\bf 27} (2024), no.~6, 1337--1355.
			
			\bibitem{DuYuLuo23}
			S.~F. Du, H. Yu and W. Luo, Regular Cayley maps of elementary abelian $p$-groups: classification and enumeration, J. Combin. Theory Ser. A {\bf 198} (2023), Paper No. 105768.
			
\bibitem{Reg-maps}
A. Gardiner, R. Nedela, J. \v Sir\'a\v n and M. \v Skoviera,
 Characterisation of graphs which underlie regular maps on closed surfaces, J. London Math. Soc. (2) {\bf 59} (1999), no.~1, 100--108.

\bibitem{Edge-t-maps}
J.~E. Graver and M.~E. Watkins, Locally finite, planar, edge-transitive graphs, Mem. Amer. Math. Soc. {\bf 126} (1997), no.~601.

			\bibitem{MR4579715}
     K. Hu, I. Kov\'acs and Y.~S. Kwon, A classification of skew morphisms of dihedral groups, J. Group Theory {\bf 26} (2023), no.~3, 547--569.

            
			\bibitem{Jajcay2003}
			R. Jajcay and J. \v Sir\'a\v n, Skew-morphisms of regular Cayley maps, Discrete Math. {\bf 244} (2002), no.~1-3, 167--179.
			
			\bibitem{KovacsKwon2016}
			I. Kov\'acs and Y.~S. Kwon, Regular Cayley maps on dihedral groups with the smallest kernel, J. Algebraic Combin. {\bf 44} (2016), no.~4, 831--847.
			\bibitem{KovacsMarusicMuzychuk2013}
            I. Kov\'acs, D. Maru\v si\v c{} and M.~E. Muzychuk, On G-arc-regular dihedrants and regular dihedral maps, J. Algebraic Combin. {\bf 38} (2013), no.~2, 437--455.

            
			\bibitem{KN2011}
            I. Kov\'acs and R. Nedela, Decomposition of skew-morphisms of cyclic groups, Ars Math. Contemp. {\bf 4} (2011), no.~2, 329--349.

            
			\bibitem{KovacsNedela2017}
I. Kov\'acs and R. Nedela, Skew-morphisms of cyclic $p$-groups, J. Group Theory {\bf 20} (2017), no.~6, 1135--1154.
  
\bibitem{Li-2006}
C.~H. Li, Finite edge-transitive Cayley graphs and rotary Cayley maps, Trans. Amer. Math. Soc. {\bf 358} (2006), no.~10, 4605--4635.


\bibitem{Li-2008}
C.~H. Li, On finite edge transitive graphs and rotary maps, J. Combin. Theory Ser. B {\bf 98} (2008), no.~5, 1063--1075.


			\bibitem{li2012cyclic}
C.~H. Li and C.~E. Praeger, On finite permutation groups with a transitive cyclic subgroup, J. Algebra {\bf 349} (2012), 117--127.

\bibitem{LiMa}
C.~H. Li and L. Ma, Locally primitive graphs and bidirect products of graphs, J. Aust. Math. Soc. {\bf 91} (2011), no.~2, 231--242.



\bibitem{LPS}
C.~H. Li, C.~E. Praeger and S.~J. Song, Locally finite vertex-rotary maps and coset graphs with finite valency and finite edge multiplicity, J. Combin. Theory Ser. B {\bf 169} (2024), 1--44.
        

\bibitem{Cay-maps} 
B. Richter, J. \v Sir\'a\v n, R. Jajcay, T. Tuker and M. Watkins, 
Cayley maps, J. Combin. Theory Ser. B {\bf 95} (2005), no.~2, 189--245.

\bibitem{Balanced-Cay}
M. \v Skoviera and J. \v Sir\'a\v n, Regular maps from Cayley graphs. I. Balanced Cayley maps, Discrete Math. {\bf 109} (1992), no.~1-3, 265--276.

\bibitem{Balanced-Cay-2}
J. \v Sir\'a\v n{} and M. \v Skoviera, Regular maps from Cayley graphs. II. Antibalanced Cayley maps, Discrete Math. {\bf 124} (1994), no.~1-3, 179--191.


			\bibitem{Zhang2015b}
			J.-Y. Zhang, A classification of regular Cayley maps with trivial Cayley-core for dihedral groups, Discrete Math. {\bf 338} (2015), no.~7, 1216--1225.
			
			\bibitem{Zhang2015a}
			J.-Y. Zhang, Regular Cayley maps of skew-type 3 for dihedral groups, Discrete Math. {\bf 338} (2015), no.~7, 1163--1172.
			
		\end{thebibliography}
	\end{document}